\documentclass[12pt,a4paper]{amsart}
\usepackage{amssymb,amsthm,mathrsfs,upgreek}
\usepackage{fullpage}
\usepackage{euscript}

\usepackage{hyperref}

\usepackage[dvipsnames]{xcolor}
\usepackage{graphicx}

\usepackage[normalem]{ulem}

\DeclareMathOperator*{\moplus}{\raisebox{-0.25ex}{\scalebox{1.2}{$\bigoplus$}}}

\newcommand{\type}[1]{$\mathrm{#1}$}
\newcommand{\Lie}{ \operatorname{{Lie}}}

\newcommand{\red}{\operatorname{{red}}}

\newcommand{\Ad}{\operatorname{{Ad}}}
\newcommand{\Aut}{\operatorname{{Aut}}}

\newcommand{\GL}{\operatorname{GL}}

\newcommand{\Gr}{\operatorname{Gr}}
\newcommand{\Pic}{\operatorname{{Pic}}}

\newcommand{\W}{\operatorname{W}}
\newcommand{\N}{\operatorname{N}}
\newcommand{\Orb}{\operatorname{O}}

\newcommand{\aaa}{\mathrm{a}}
\newcommand{\s}{\mathrm{s}}
\newcommand{\m}{\mathrm{m}}

\newcommand{\g}{\mathrm{g}}
\newcommand{\z}{\mathrm{z}}
\newcommand{\rt}{\mathrm{t}}
\newcommand{\rrr}{\mathrm{r}}
\newcommand{\rc}{{\mathrm{c}}}

\newcommand{\balpha}{{\boldsymbol\alpha}}

\newcommand{\Ga}{{\mathbb G}_{\mathrm{a}}}
\newcommand{\Gm}{{\mathbb G}_{\m}}
\newcommand{\ZZ}{{\mathbb{Z}}}

\newcommand{\PP}{{\mathbb{P}}}

\newcommand{\Sym}{{\mathfrak{S}}}

\newcommand{\fD}{{\mathfrak{D}}}

\newcommand{\fh}{{\mathfrak{h}}}
\newcommand{\fg}{{\mathfrak{g}}}

\newcommand{\fH}{{\mathfrak{H}}}
\newcommand{\fp}{{\mathfrak{p}}}
\newcommand{\fb}{{\mathfrak{b}}}

\newcommand{\cL}{\mathcal{L}}

\newcommand{\OOO}{{\mathscr{O}}}

\newcommand{\EEE}{{\mathscr{E}}}
\newcommand{\UUU}{{\mathscr{U}}}

\renewcommand{\emptyset}{\varnothing}
\newcommand{\xref}[1]{\textup{\ref{#1}}}

\makeatletter
\@addtoreset{equation}{subsection}
\@addtoreset{figure}{subsection}
\makeatother
\renewcommand{\theequation}{\arabic{section}.\arabic{subsection}.\arabic{equation}}

\newtheorem{mthm}{Theorem}

\swapnumbers
\newtheorem{thm}[subsection]{Theorem}

\newtheorem{scor}[equation]{Corollary}
\newtheorem{lem}[subsection]{Lemma}
\newtheorem{prop}[subsection]{Proposition}
\theoremstyle{definition}

\newtheorem{srem}[equation]{Remark}
\newtheorem{sit}[equation]{}

\title[Fano--Mukai fourfolds: Automorphisms]{Fano--Mukai fourfolds of genus $10$ and their automorphism groups}
\author{Yuri Prokhorov}\thanks{
The research of the first author
 was partially supported by the HSE University Basic Research Program.
}

\author{Mikhail Zaidenberg}
\address{\emph{Yuri Prokhorov}\newline
Steklov Mathematical Institute, Moscow, Russian Federation
\newline
National Research University Higher School of Economics,
Russian Federation
\newline
Department of Algebra, 
Moscow State Lomonosov University, Russian Federation
}
\email{prokhoro@mi-ras.ru}
\address{\emph{Mikhail Zaidenberg}\newline Universit\'e Grenoble Alpes, CNRS, Institut Fourier, F-38000 Grenoble, France}
\email{Mikhail.Zaidenberg@univ-grenoble-alpes.fr}

\begin{document}

\begin{abstract} 
The automorphism groups of the Fano--Mukai fourfold of genus $10$ were studied in our previous paper \cite{PZ18}. 
In particular, we found in \cite{PZ18} the neutral components of these groups. 
In the present paper we finish the description of the discrete parts. 

Up to isomorphism, there are three special Fano--Mukai fourfold of genus $10$ with the
automorphism groups $\GL_2(\Bbbk)\rtimes\ZZ/2\ZZ$, $(\Ga\times\Gm)\rtimes\ZZ/2\ZZ$ and $\Gm^2\rtimes\ZZ/6\ZZ$, respectively. 
For any other Fano--Mukai fourfold $V$ of genus $10$ one has $\Aut(V)=\Gm^2\rtimes\ZZ/2\ZZ$.
\end{abstract}
\date{} 
\maketitle

\setcounter{tocdepth}{2}\tableofcontents

\section*{Introduction} 
Our base field is 
an algebraically closed field $\Bbbk$ of characteristic zero.
We use the notation of \cite{PZ18}. Let $V_{18}$ 
be a Fano--Mukai fourfold of genus $10$ and degree $18$ half-anticanonically embedded in $\PP^{12}$.
Recall that these fourfolds form a one-dimensional family, see e.g. (\cite{KapustkaRanestad2013}, \cite[Remark~13.4]{PZ18} 
or Corollaries~\ref{cor:orbits-isom.classes} and~\ref{cor:max-torus-new} below. 
This family contains two special members, namely, $V^{\s}_{18}$ with 
$\Aut^0(V)=\GL_2(\Bbbk)$ and $V^{\aaa}_{18}$ with $\Aut^0(V^{\aaa}_{18})=\Ga\times\Gm$ \cite[Theorem~1.3(i),(ii)]{PZ18}. 
For the general member $V^{\g}_{18}\notin \{V^{\aaa}_{18},\, V^{\s}_{18}\}$ one has $\Aut^0(V^{\g}_{18})=\Gm^2$ \cite[Theorem~1.3(iii)]{PZ18}.

In the present paper we complete the description of the automorphism groups $\Aut(V)$ of the Fano--Mukai fourfolds $V$ of genus $10$ in the general case $V=V^{\g}_{18}$, 
which is done in \cite{PZ18} only partially. 
Summarizing the results of \cite{PZ18} and the ones of Theorem~\ref{thm:aut} and Corollary~\ref{mthm:2:iii} below we get the following

\begin{mthm}
\label{thm:aut0} 
Let $V$ be a Fano--Mukai fourfold $V$ of genus $10$. 
Then the following assertions hold.
\begin{enumerate}
\item
\label{thm:aut0:i} 
The group $\Aut(V)$ is isomorphic to one of the following:
\renewcommand{\theequation}{\Alph{mthm}.\arabic{equation}}
\begin{eqnarray}
&&
\label{thm:aut0:GL2} 
\GL_2(\Bbbk)\rtimes \ZZ/2\ZZ,
\\
&&
\label{thm:aut0:G2-Z6}
\Gm^2\rtimes \ZZ/6\ZZ,
\\
&&
\label{thm:aut0:G2-Z2}
\Gm^2\rtimes \ZZ/2\ZZ,
\\
&&
\label{thm:aut0:GaGm}
(\Ga\times\Gm)\rtimes \ZZ/2\ZZ.
\end{eqnarray}
In cases \eqref{thm:aut0:GL2}, \eqref{thm:aut0:G2-Z2}, and \eqref{thm:aut0:GaGm} 
the generator of the subgroup $\ZZ/2\ZZ$ acts on $\Aut^0(V)$ via the involution $g\mapsto (g^{\rt})^{-1}$ where 
we let $g^{\rt}:=g$ in the case of an abelian group $\Aut^0(V)$.
In case~\eqref{thm:aut0:G2-Z6} the generator of $\ZZ/6\ZZ$
acts on $\Gm^2$ via an automorphism of order $6$ from $\Aut(\Gm^2)\cong\GL_2(\ZZ)$. 
\item
\label{thm:aut0:ii} 
In cases \eqref{thm:aut0:GL2}, \eqref{thm:aut0:G2-Z6}, \eqref{thm:aut0:GaGm} the corresponding variety $V$
is unique up to isomorphism.
In case \eqref{thm:aut0:G2-Z2} the variety $V$ vary in a one-parameter family.
\item
\label{thm:aut0:iv} 
Let $G$ be the simple algebraic group of type \type{G_2} and $\Omega\subset\PP^{13}$ be the adjoint variety of $G$.
Under the Mukai realization of $V$ 
as a hyperplane section of $\Omega$, 
any element of $\Aut(V)$ is induced by an automorphism of $\Omega$ and 
the group $\Aut(V)$ coincides with the stabilizer of $V$ in $G$.
\end{enumerate}
\end{mthm}

Besides, in Proposition~\ref{mthm:2:iii} we describe the action of the group $\Aut(V)$ on the set of cubic cones in $V$, 
that is, the cones over the rational twisted cubic curves. Except in the case \eqref{thm:aut0:GL2} this set is finite, 
hence each cubic cone is stable under the $\Aut^0(V)$-action. In Proposition~\ref{mthm:2:iii} the action of the 
component group $\Aut(V)/\Aut^0(V)$ on this set is determined. 

Some other families of Fano varieties demonstrate similar behavior of the automorphism groups. For instance, this concerns
the Fano threefolds of degree~$22$ and Picard number~$1$ \cite{Kuznetsov-Prokhorov, Prokhorov-1990c} 
and the Fano threefolds of degree~$28$ and
Picard number~$2$ \cite[\S~9]{Prz-Ch-Shr:19}.
The automorphism groups of generic Fano-Mukai varieties of genus $g\in {7,\ldots,10}$ were described in a recent preprint~\cite{DM21}.
In the appendix of \emph{loc. cit.} it was also shown that a generic Fano threefolds of degree~$22$ and Picard number~$1$ has trivial automorphism group.

The paper is organized as follows. In Section~\ref{sec:prelim}
we gather necessary preliminaries, in particular, some results from \cite{Muk89, Kuz06, PZ18, PZ20}. Theorem~\ref{thm:aut0} is proven in
Section~\ref{sec:aut}.

Our paper has grown up from an earlier version of the preprint~\cite{PZ20}, 
where the major results were already present. 

\subsubsection*{Acknowledgment.} 
The authors are grateful to Alexander Kuznetsov for helpful discussions
and Alexander Perepechko for a careful reading of the preliminary version of the manuscript and valuable remarks. Our thanks are due also to the referees for their remarks that helped us to improve the presentation. 

\section{Preliminaries}
\label{sec:prelim}
In this section we recall some facts from \cite{PZ18} and~\cite{PZ20} used in the sequel. 
Throughout the paper we let $V=V_{18}$ be a 
Fano--Mukai fourfold of genus $10$ half-anticanonically embedded in $ \PP^{12}$.

\subsection{The Mukai realization}
\label{ss:Mukai}
Consider the Lie algebra $\fg=\fg_2$ of the simple complex algebraic group $G$ of type \type{G_2}, the projective space $\PP\fg=\PP^{13}$, 
and the projective adjoint representation $\Ad(G)$ on $\PP\fg$. 
Let $P\subset G$ be a parabolic subgroup such that $\Omega=G/P\subset \PP^{13}$  
is the minimal nilpotent orbit of $G$, called
the \emph{adjoint variety},
see \cite{Muk89}.
Fix a maximal torus $T\subset P$ of $G$, and let $\fh=\Lie(T)$ be the corresponding Cartan subalgebra of  $\fg$. 
Choose a basis of simple roots $(\balpha_1, \balpha_2)$ of the root system $\Delta\subset \fh^\vee$ with 
$\balpha_1\in\Delta_\s$ and $\balpha_2\in\Delta_\ell$ where $\Delta_\ell$ and $\Delta_\s$ stand for the subsets of long 
and short roots, respectively. We identify $\fg$ with the dual vector space $(\fg)^\vee$ via the duality on $\fg$ defined by 
the Killing form and we identify the dual projective space $\PP\fg=\PP^{13}$ with its dual $(\PP^{13})^\vee$. Under this identification
the $\Ad(G)$-orbits coincide with the corresponding $\Ad^*(G)$-orbits \cite[Section~1.3]{CMG93}. 
Consider the root space decomposition
\begin{equation*}
\fg=\fh
\moplus\Bigl(\moplus_{\balpha\in\Delta_\ell}
\fg_\balpha\Bigr)
\moplus\Bigl(\moplus_{\balpha\in\Delta_\s}
\fg_\balpha\Bigr)
\end{equation*}
where $\fg_\balpha= \Bbbk e_\balpha$ for 
$\balpha\in\Delta$ is the root subspace generated by the root vector $e_\balpha$. 
One can choose for $P$
the parabolic subgroup of $G$ with Lie algebra 
\begin{equation*}
\fp=\fh
\moplus\Bigl(\moplus_{\balpha\in\Delta^+}
\fg_\balpha\Bigr)
\moplus\fg_{-\balpha_1}=
\fb\moplus\fg_{-\balpha_1}
\end{equation*}
where $\fb$ is the corresponding Borel subalgebra of $\fg$.

\begin{sit}{\it Mukai's vector bundle.}
\label{sit:Mukai-presentation} 
Recall \cite{Muk88, Muk89, Muk92} that any Fano--Mukai fourfold of genus $10$ can be presented as a hyperplane section $V(h)=\Omega\cap h^\bot$ 
for a nonzero element $h\in\fg=\Lie(G)$ where $h^\bot$ stands for the projectivization of the orthogonal complement of $\Bbbk h$ in $\fg$ 
with respect to the Killing form. Let us recall this construction in more detail. 

Any Fano--Mukai fourfold $V$ of genus $10$ carries a stable vector bundle $\EEE$ of rank two such that
\begin{itemize}\item
$\EEE$ is generated by global sections;
\item
$\rc_1(\EEE)$ is an ample generator of $\Pic(V)$; 
\item
$\dim H^0(V,\EEE) = 7$.
\end{itemize}
The vector bundle $\EEE$ is unique up to isomorphism, see \cite[Section~2]{Muk89}, and also 
\cite[Proposition~B.1.5]{Kuznetsov-Prokhorov-Shramov} where the uniqueness is proven in the case of Fano threefolds. 
The proof works mutatis mutandis in our case as well.
\end{sit}
Denote $W:=H^0(V,\EEE)^\vee$.
The vector bundle $\EEE$ determines an embedding
\[
i: V \hookrightarrow \Gr(2,W)
\]
such that $\EEE=i^* \UUU^\vee$ where $\UUU$ is the tautological rank-2 vector bundle on the Grassmannian $\Gr(2,W)=\Gr(2,7)$. 
Consider the  natural map
\[
\Psi:  H^0\big(\Gr(2,W),\, \UUU^{\perp}(1)\big) \longrightarrow 
 H^0\big(V,\,i^* \UUU^{\perp}(1)
\big),
\]
where 
$\UUU^{\perp}:=\ker \big(W^\vee \otimes \OOO_{\Gr(2,W)} \to \UUU^\vee\big)$.
The kernel of $\Psi$ is one-dimensional and generated by a section 
\[
s\in H^0\big(\Gr(2,W),\, \UUU^{\perp}(1)\big)=\wedge^3  W^\vee
\]
so that $G$ is the stabilizer of $s$ in $\GL(W)$.
Hence the zero locus $Z=Z(s)$ admits an action of $G$; this is a $G$-homogeneous space isomorphic to $\Omega$ 
(see  \cite[Section~8]{Kuz06} for details). 
The image $i(V)$ is contained in $Z$. 
This defines embeddings 
\begin{equation}
\label{eq:Z} 
V\stackrel{i} {\hookrightarrow} Z \subset \Gr(2,W)
\end{equation}
such that $i(V)$ is a hyperplane section of $Z$ under the Pl\"{u}cker embedding of $\Gr(2,W)$.

\subsection{The orbits of the adjoint action of $G$}
\label{ss:orbits} 
Any orbit of the projective adjoint representation is the image of an orbit of the adjoint representation. We keep the terminology ``regular orbit'', 
``semisimple orbit'', ``nilpotent orbit'', etc. for the images in $\PP^{13}$ of the $\Ad(G)$-orbits in $\fg$ 
consisting of regular, semisimple, nilpotent, etc. elements, respectively. 
For the following facts see \cite[Theorem~8.25]{Tev05}.

\begin{prop}
\label{prop:summary} 
$\,$
\begin{enumerate}
\item
\label{prop:summary:i}
The fivefold $\Omega$ is the unique closed nilpotent orbit of the $\Ad(G)$-action on $\PP\fg=\PP^{13}$. 
It passes through the points 
$\PP \fg_\balpha$ which correspond to the long roots 
$\balpha\in\Delta_\ell$.

\item
\label{prop:summary:ii} 
The dual projective variety $D_\ell$ of $\Omega$ is an irreducible hypersurface in $(\PP^{13})^\vee=\PP^{13}$
given by a homogeneous polynomial $\phi_\ell$ of degree six on $\fg$ such that 
$\phi_\ell|_{\fh}=\prod_{\balpha\in\Delta_\ell} \balpha$.
Thus, the fourfold $V(h)=\Omega\cap h^\bot$ is singular if and only if $\PP h\in D_\ell$ where $\PP h\in\PP\fg$ is the  line in $\fg$ through $h$.

\item
\label{prop:summary:iii}
For a short root $\balpha\in\Delta_\s$ the orbit $\Omega_\s=\Ad(G).\PP \fg_\balpha\subset\PP^{13}$ 
does not depend on the choice of $\balpha\in\Delta_\s$.

\item
\label{prop:summary:iv} 
Let $D_\s\subset(\PP^{13})^\vee=\PP^{13}$ be the projective variety dual to the orbit closure $\overline{\Omega_\s}$. Then $D_\s$
is an irreducible hypersurface in $\PP^{13}$ 
given by a homogeneous polynomial $\phi_\s$ of degree six on $\fg$ such that 
$\phi_\s|_{\fh}=\prod_{\balpha\in\Delta_\s} \balpha$.
\end{enumerate}
\end{prop}

\begin{prop} 
\label{lem:classification-of-orbits} 
For the adjoint representation $\Ad(G)$ on $\PP\fg=\PP^{13}$ the following hold.
\begin{enumerate}
\item
\label{lem:classification-of-orbits:i} 
\textup(\cite[Corollary~2.1.13]{CMG93}\textup)
The complement $\PP\fg\setminus (D_\ell\cup D_\s)$ is the image in $\PP\fg$ of the 
set of regular semisimple elements of $\fg$, and $D_\ell\cap D_\s$ is the image of the nilpotent cone of $\fg$.
\item
\label{lem:classification-of-orbits:iv} 
Both $\Omega$ and $\Omega_{\s}$
are nilpotent orbits contained in $D_\ell\cap D_\s$. 
\item
\label{lem:classification-of-orbits:ii} 
\textup(\cite[Lemma~1]{KapustkaRanestad2013}\textup) $D_\s\setminus D_\ell$ is the union of two orbits 
$\Omega^{\aaa}$ of dimension $12$ and $\Omega^{\s\rrr}$ of dimension $10$ where
\begin{itemize}
\item
$\Omega^{\aaa}$ is a non-semisimple regular orbit open and dense in $D_\s\setminus D_\ell$ and 
\item
$\Omega^{\s\rrr}$ is a semisimple subregular orbit closed in $D_\s\setminus D_\ell$.
\end{itemize}
\item
\label{lem:classification-of-orbits:iii} 
\textup(\cite[Lemma~1]{KapustkaRanestad2013}\textup) Let $\cL=\langle D_\ell,D_\s\rangle$ 
be the pencil of sextic hypersurfaces in $\PP^{13}$ generated by $D_\ell$ and $D_\s$. 
Then for any member $D_t$ of $\cL$ different from $D_\ell$ and $D_\s$
the complement $D_t\setminus D_\ell$ is a regular semisimple orbit of $\Ad(G)$.
\end{enumerate}
\end{prop}

\begin{proof}
\ref{lem:classification-of-orbits:iv} Both $\Omega$ and $\Omega_{\s}$ 
are the orbits of root subspaces consisting of nilpotent elements, 
see Proposition~\ref{prop:summary}\ref{prop:summary:i},~\ref{prop:summary:iii}. 
The inclusion $\Omega\cup \Omega_{\s}\subset D_\ell\cap D_\s$ follows from~\ref{lem:classification-of-orbits:i}.
\end{proof}

\begin{prop}\textup(\cite[7.8.3-7.8.4 and Theorem~12.1]{PZ18}\textup)\label{prop:aut0}
Let $V=V(h)$ for $h\in\fg\setminus\{0\}$. Then 
\[
\Aut^0(V)\cong
\begin{cases}
\Gm^2&\text{if and only if $h^\bot\in (\PP\fg)^\vee\setminus (D_\ell\cup D_\s)$};
\\
\GL_2(\Bbbk)&\text{if and only if $\PP h\in \Omega^{\s\rrr}$};
\\
\Ga\times\Gm&\text{if and only if $\PP h\in \Omega^{\aaa}$}
\end{cases}
\]
where $\Omega^{\s\rrr}\cup\Omega^{\aaa}=D_\s\setminus D_\ell$, 
see Proposition~\xref{lem:classification-of-orbits} \xref{lem:classification-of-orbits:iii}.
\end{prop}

\begin{prop}
\label{prop:Omega} 
$\,$
\begin{enumerate}\item
\label{lem:aut-omega}
One has $\Aut(\Omega)=G$.
\item
\label{fixed-points} 
Given a maximal torus $T$ of $G$, the $T$-action on $\Omega$ 
has exactly six $T$-fixed points. These are the points $\PP \fg_\balpha\in\Omega$ corresponding to the long roots 
$\balpha\in\Delta_\ell$, see Proposition~\xref{prop:summary}\xref{prop:summary:i}.
\end{enumerate}
\end{prop}

\begin{proof}
\ref{lem:aut-omega} follows from ~\cite[Section 3.3, Theorems 2 and 3]{Akh95}.

\ref{fixed-points} The $T$-fixed point set in $\PP\fg=\PP^{13}$ is the projectivization 
of the union of eigenspaces of $T$ acting on $\PP\fg$. Due to the Cartan decomposition,
this set consists of the projective line 
$\PP\fh$ and the twelve isolated points $\PP\fg_\balpha$ where $\balpha\in\Delta_\ell\cup\Delta_\s$.
The six fixed points $\PP\fg_\balpha$ which correspond to the short roots $\balpha\in\Delta_\s$ 
lie on the nilpotent orbit $\Omega_\s\neq\Omega$, see Proposition~\ref{lem:classification-of-orbits}\ref{lem:classification-of-orbits:iv}. 
The projective line $\PP\fh$ parameterizes the semisimple orbits. Hence, $\PP\fh\cap\Omega=\emptyset$, cf.~\cite[Theorem~2.2.4]{CMG93}. 
The remaining six fixed points $\PP\fg_\balpha$, $\balpha\in\Delta_\ell$, are points of $\Omega$, 
see Proposition~\ref{prop:summary}\ref{prop:summary:i} and~\ref{prop:summary:iii}.
\end{proof}

\subsection{Cubic cones}
\label{sit:Cubic cones}
The Fano--Mukai fourfolds $V=V_{18}$ are classified in three types according to the group $\Aut^0(V)$. The latter
algebraic group is isomorphic to one of the following groups:
\begin{equation}
\label{aeq:Aut0}
\Gm^2,\qquad \Ga\times\Gm,\qquad \GL_2(\Bbbk).
\end{equation} 
This classification reflects the geometry of $V$, namely, the number of cubic cones on $V$. 
Following~\cite{KapustkaRanestad2013} we call a \emph{cubic cone} 
the cone over a rational twisted cubic curve. 

\begin{lem}[{\cite[Lemma~1.8]{PZ20}}, {\cite[Section~12]{PZ18}}] 
\label{lem:cubic-cones:2} 
$\,$
\begin{enumerate} 
\item
\label{lem:cubic-cones:2:cycle} 
If $\Aut^0(V)=\Gm^2$ then $V$ contains exactly $6$ cubic cones. 
These cones form a cycle in which the neighbors are the cones with a common ruling, 
and the pairs of opposite members
of the cycle correspond to the pairs of disjoint cubic cones.
\item
\label{lem:cubic-cones:2:iii} 
If $\Aut^0(V)=\Ga\times\Gm$ then $V$ contains exactly $4$ cubic cones. 
These cones form a chain in which the neighbors are the cones with a common ruling,
and the pair of end vertices
corresponds to the unique pair of disjoint cones.
\item
\label{lem:cubic-cones:2:i} 
If $\Aut^0(V)=\GL_2(\Bbbk)$ then the Hilbert scheme of cubic cones in $V$ 
is a disjoint union of two projective lines and of two isolated points. 
The cubic cones represented by these isolated points are disjoint, 
and this is the only pair of $\Aut^0(V)$-invariant cubic cones on $V$.
\end{enumerate}
\end{lem}

\section{Automorphism groups of Fano--Mukai fourfolds of genus $10$}
\label{sec:aut}
In this section we provide a proof of Theorem~\ref{thm:aut0} from the Introduction. 
In \cite[Theorem~1.3(i),(ii)]{PZ18} we established already that
\[
\Aut(V^{\s}_{18})=\GL_2(\Bbbk)\rtimes \ZZ/2\ZZ\quad\text{and}\quad \Aut(V^{\aaa}_{18})=(\Ga\times\Gm)\rtimes \ZZ/2\ZZ.
\] 
Moreover, these varieties are unique up to isomorphism.
The generator of $\ZZ/2\ZZ$ acts on $\Aut^0(V)$ via the involution $g\mapsto ({g^{\rt}})^{-1}$ 
in the first case and via $g\mapsto g^{-1}$ in the second. In the general case, that is, for $V=V^{\g}_{18}$, we have $\Aut^0(V)=\Gm^2$ and
\begin{equation*}
\Aut(V)\subset \Gm^2\rtimes \ZZ/6\ZZ,
\end{equation*}
see \cite[Theorem~1.3(iii) and Lemma 11.4]{PZ18}.
	The following Theorem~\ref{mthm:2} 
	completes these results and gives a proof of Theorem~\ref{thm:aut0}\ref{thm:aut0:i}-\ref{thm:aut0:ii}; 
	see Corollary~\ref{cor:new} for the proof of Theorem~\ref{thm:aut0}\ref{thm:aut0:iv}.

\begin{thm}
\label{mthm:2}
\label{thm:aut} $\,$
\begin{enumerate}
\item
\label{mthm:2:i} 
There exists a unique, up to isomorphism, Fano--Mukai fourfold $V^{\rrr}$ of genus $10$ with
$\Aut(V^{\rrr})=\Gm^2\rtimes \ZZ/6\ZZ$. \footnote{This answers a question in~\cite[Remark~11.5.2]{PZ18}.}
Any generator of $\ZZ/6\ZZ$ acts on $\Gm^2$ via an automorphism of order $6$ 
from $\Aut(\Gm^2)\cong\GL_2(\ZZ)$. 

\item
\label{mthm:2:ii} 
For any Fano--Mukai fourfold $V$ 
of genus $10$
with $\Aut^0(V)=\Gm^2$ non-isomorphic to
$V^{\rrr}$ one has $\Aut(V)=\Gm^2\rtimes \ZZ/2\ZZ$. 
The generator of $\ZZ/2\ZZ$ acts on $\Gm^2$ via $g\mapsto g^{-1}$.
\end{enumerate}
\end{thm}

The proof of Theorem~\ref{mthm:2} 
is postponed until subsection~\ref{ss:mthm:2}. 

\subsection{Proof of Theorem~\ref{thm:aut0}\ref{thm:aut0:iv}}

The following proposition is an analog of \cite[Theorem~0.9]{Muk88} 
where a similar fact is stated for Fano 3-folds of degree 18 with Picard number one.

\begin{prop}
\label{prop:ix}
Let $V$ and $V'$ be two Fano--Mukai fourfolds of genus $10$ which are realized as hyperplane sections of 
$\Omega\subset \PP^{13}$. If $V\cong V'$ 
then $V$ and $V'$ are equivalent under the adjoint action of $G$ on $\Omega$. Moreover, any isomorphism 
$V\stackrel{\cong}{\longrightarrow} V'$ admits an extension to an automorphism of $\Omega$ from $\Ad(G)|_{\Omega}$. 
\end{prop}

\begin{proof} 
Fix an isomorphism $\varphi\colon V\stackrel{\cong}{\longrightarrow} V'$. Let $\EEE$ 
($\EEE'$, respectively) 
be the Mukai vector bundle on $V$ (on $V'$, respectively). Letting $W=H^0(V,\EEE)^\vee$ and $W'=H^0(V',\EEE')^\vee$ consider the corresponding embeddings 
\[
i: V \hookrightarrow Z\hookrightarrow\Gr(2,W),\quad\text{resp.}\quad i': V' \hookrightarrow Z'\hookrightarrow\Gr(2,W')
\] 
where $Z\cong Z'\cong\Omega$, see~\eqref{eq:Z}. 
The Mukai vector bundles $\varphi^*\EEE$ and $\EEE'$ 
over $V'$ are isomorphic, cf.~\cite[Proposition~5.1]{Muk88} and \cite[Proposition~B.1.5]{Kuznetsov-Prokhorov-Shramov}. 
Hence,
we may suppose $\EEE'=\varphi^*\EEE$. 
Under this assumption, $\varphi$ induces isomorphisms
\begin{itemize}
\item
$W\cong_{\Bbbk} W'$;
\item
$(\Gr(2,W),Z,i(V))\cong (\Gr(2,W'), Z',i'(V'))$.
\end{itemize}
Choosing isomorphisms $Z\cong\Omega$ and $Z'\cong\Omega$ we obtain embeddings
$j\colon V\hookrightarrow\Omega$ and $j'\colon V'\hookrightarrow\Omega$ and an automorphism of $\Omega$ 
which sends $j(V)$ onto $j'(V')$. Now the assertion follows from the equality
$\Aut(\Omega)=G$, see Proposition~\ref{prop:Omega}\ref{lem:aut-omega}.
\end{proof}

Combining this with the existence of a Fano--Mukai presentation~\eqref{sit:Mukai-presentation}
we deduce the following corollary.

\begin{scor}
\label{cor:orbits-isom.classes}
There is a bijection between the isomorphism classes 
of the Fano--Mukai fourfolds of genus $10$ and 
the $G$-orbits in $(\PP^{13})^\vee\setminus D_\ell$.
\end{scor}

\begin{proof}
Since $\Omega\subset\PP^{13}$ is linearly non-degenerate, the smooth hyperplane sections $V(h)=\Omega\cap h^\bot$ 
of $\Omega$ are parameterized by $h^\bot\in (\PP^{13})^\vee\setminus D_\ell$, see Proposition~\ref{prop:summary}\ref{prop:summary:ii}. 
The assertion follows now from Proposition~\ref{prop:ix}.
\end{proof}

The next corollary 
proves Theorem~\ref{thm:aut0}\ref{thm:aut0:iv}.

\begin{scor}
\label{cor:new}
Under the Mukai realization of $V$ as a hyperplane section of the $G$-homogeneous fivefold $\Omega$ 
the group $\Aut(V)$ coincides with the stabilizer of $V$ in $G$. 
\end{scor}

\begin{proof}
Due to Lemma~7.8 in~\cite{PZ18} the stabilizer of $V$ in $G=\Aut(\Omega)$ 
(see Proposition~\ref{prop:Omega}\ref{lem:aut-omega}) acts effectively on $V$. Now 
Proposition~\ref{prop:ix} implies the assertion. 
\end{proof}

\begin{srem}
\label{rem:old-rem}
This corollary extends~\cite[Theorem~1.3]{PZ18}. It also disproves Remark~15.5 in~\cite{PZ18}. 
Indeed, in this remark
we erroneously considered the stabilizer in $G$ of a nonzero vector $h\in\fg$ 
instead of the stabilizer in $G$ of the line $\Bbbk h\in\PP\fg=\PP^{13}$. 
\end{srem}

\subsection{The automorphism groups as stabilizers}

The following proposition and its corollary will be important in the proof of Theorem~\ref{mthm:2}. 

\begin{prop}
\label{lemma:max-torus} 
Consider a smooth Fano--Mukai fourfold $V=V(h)=\Omega\cap h^\bot$ where $h\in\fg$ 
is nonzero, see \xref{sit:Mukai-presentation}. Let $\fh\subset\fg$ be a Cartan subalgebra. 
\begin{enumerate} 
\item 
\label{lemma:max-torus:i} 
The group $\Aut(V)$ is reductive if and only if the element $h\in\fg$ is semisimple.
\item
\label{lemma:max-torus:ii} 
If the group $\Aut(V)$ is reductive then 
there exist $h'\in\fh\setminus \{0\}$ and $g\in G$ such that $V=\Ad(g)(V(h'))$.
\end{enumerate}
\end{prop}

\begin{proof}
\ref{lemma:max-torus:i} 
Recall~\cite[Theorem~III.5.3]{Serre66} 
that any Cartan subalgebra of a semisimple Lie algebra coincides with its centralizer. It follows that
an element $h\in\fg\setminus \{0\}$ is semisimple if and only if the stabilizer
\footnote{Abusing notation, in the sequel we write $G_{h}$ meaning $G_{\PP h}$.}
$G_{\PP h}=\Aut(V)$ contains a maximal torus of $G$. 
Due to \eqref{aeq:Aut0} the latter occurs if and only if
$\Aut(V)$ is reductive. 

\ref{lemma:max-torus:ii} follows from the facts that 
any semisimple element $h\in\fg$ is contained in some 
Cartan subalgebra and any two Cartan subalgebras are conjugate.
\end{proof}

\begin{scor}
\label{cor:max-torus-new} 
Let $\fH$ be the family 
of hyperplane sections $V(h)$ of $\Omega$, $h\in \fh\setminus\{0\}$, parameterized
by the projective line $\PP\fh\subset\PP\fg$ where $\fh\subset\fg$ is a Cartan subalgebra. Then 
any Fano--Mukai fourfold $V$ of genus $10$ with a reductive group $\Aut(V)$ is isomorphic to a member of $\fH$. 
\end{scor}

\begin{srem}
\label{rem:stab} 
Let $G_{\fh}\subset G$ be the stabilizer of the line $\PP\fh\subset \PP\fg$
and let $G_{h, \fh}\subset G_{\fh}\subset G$ be the stabilizer of the flag $\PP h\subset \PP\fh\subset \PP\fg$ where
$h\in \fh\setminus \{0\}$. 
It is clear that $G_{\fh}\subset G$ coincides with the normalizer ${\N}_{G}(T)$ of the maximal torus $T$ in $G$ with ${\Lie}(T)=\fh$. There is a splitting
\begin{equation*}
G_{\fh}=\N_{G}(T)\cong T\rtimes \W
\end{equation*} where
$\W={\N}_{G}(T)/T$ stands for the Weyl group of $G$;
see \cite[Theorem~A]{AH17}.
\end{srem}

\begin{lem} 
\label{lemma:stabilizer} 
Consider a Fano--Mukai fourfold $V=V(h)$ where $h\in\fh\setminus\{0\}$. Assume $\Aut^0(V)=T\cong\Gm^2$.
One has
\begin{equation*}
\Aut(V)=G_{h, \fh}\subset G_{\fh}\cong T\rtimes \W.
\end{equation*}
\end{lem}

\begin{proof} 
The stabilizer $G_h$ of the point $\PP h\in \PP\fg$ coincides with $\Aut(V)$,
see Corollary~\ref{cor:new}.
By our assumption, $(G_h)^0=\Aut^0(V)=T$ is a maximal torus of $G$.
Clearly, $G_h\supset G_{h, \fh}$.
On the other hand, since $(G_h)^0=T$, the Cartan subalgebra of $\fg=\Lie(T)$ 
is preserved under the action of $G_h$, hence $G_h\subset G_{h, \fh}$. 
It follows that 
\[
\Aut(V)=G_h=G_{h, \fh}.\qedhere
\] 
\end{proof}

\begin{srem}
\label{rem:sym} Since $\fh$ is abelian, the torus $T$ acts trivially on $\PP\fh$. 
The Weyl group $\W$ is isomorphic to the dihedral group $\fD_6\cong \Sym_3\times (\ZZ/2\ZZ)$ where $\Sym_3$ stands for the symmetric group of rank three. The center of $\fD_6$ acts via the central symmetry on $\fh$ 
and trivially on $\PP\fh$, and no other element of $\fD_6$ does. Hence, the $\N_{G}(T)$-action on $\fh$ induces an effective $\Sym_3$-action 
on $\PP\fh\cong\PP^1$.
The orbits of the latter action
are as follows.
\end{srem}

\begin{lem} 
\label{lemma:stabilizer-S3} 
Any effective $\Sym_3$-action on $\PP^1$ has exactly $3$ orbits with nontrivial stabilizers, namely,
two orbits $\Orb_{\ell}$, $\Orb_{\s}$ of length $3$ and one orbit $\Orb_{\rrr}$ of length $2$.
\end{lem}

\begin{proof} 
This follows from the Riemann-Hurwitz formula
applied to the Galois covering $\PP^1\to \PP^1/\Sym_3\cong\PP^1$.
\end{proof}

\begin{lem} 
\label{lemma:intersections}
Up to permutation of $\Orb_{\ell}$ and $\Orb_{\s}$ we may assume that 
$\Orb_{\ell}=(D_{\ell}\cap \PP\fh)_{\red}$ and $\Orb_{\s}=(D_{\s}\cap \PP\fh)_{\red}$.
The orbit $\Orb_{\rrr}$ is given by vanishing of the invariant Killing quadratic form on $\fh$.
\end{lem}

\begin{proof} 
The Killing quadratic form on $\fh$ is 
$\W$-invariant. Hence, its zero set in $\PP\fh$
is $\Sym_3$-invariant. 
Since the Killing form on $\fh$ is nondegenerate,
this set coincides with the orbit $\Orb_{\rrr}$ of length $2$.

The Weyl group acts on the plane $\fh^\vee$ via the standard representation 
$\fD_6\hookrightarrow \GL_2(\Bbbk)$. This group is generated by reflexions and preserves the sets 
$\Delta_\ell$ and $\Delta_\s$ of long and short roots, respectively. 
The mirrors of reflexions from $\W$ are  the  vector lines in $\fh^\vee$ through the pairs of opposite roots from $\Delta_\ell$ and $\Delta_\s$. 
These 6 lines are projected to 
the triples of points $\PP\Delta_{\s}$ and $\PP\Delta_{\ell}$ of the projective line $\PP\fh=\PP\fh^\vee$. Being
$\Sym_3$-invariant, these triples coincide with the orbits $\Orb_{\ell}$ and $\Orb_{\s}$. 
By Lemma~\ref{lemma:stabilizer-S3} we may assume that 
$\Orb_{\ell}=\PP\Delta_{\s}$ and $\Orb_{\s}=\PP\Delta_{\ell}$. 

Consider on $\fh$ the polynomials 
\[
\psi_{\ell}=\prod_{\balpha\in\Delta_{\ell}}\balpha=-\Bigl(\prod\nolimits_{\balpha\in\Delta_{\ell}^+}\balpha\Bigr)^2,\qquad 
\psi_{\s}=\prod_{\balpha\in\Delta_{\s}}\balpha=-\Bigl(\prod\nolimits_{\balpha\in\Delta_{\s}^+}\balpha\Bigr)^2
\] 
where $\Delta_{\ell}^+,\Delta_{\s}^+$ are the sets of positive long and short roots, respectively. These polynomials are $\W$-invariant
and vanish with multiplicity $2$ on $\PP\Delta_{\s}=\Orb_{\ell}$ and $\PP\Delta_{\ell}=\Orb_{\s}$, respectively. 
Indeed, there exist pairs of Killing-orthogonal short and long roots. 
By Proposition~\ref{prop:summary}\ref{prop:summary:ii},\ref{prop:summary:iv} one has $\psi_{\ell}=\phi_{\ell}|_{\fh}$ and $\psi_{\s}=\phi_{\s}|_{\fh}$. 
It follows that $\Orb_{\ell}=(D_{\ell}\cap \PP\fh)_{\red}$ and $\Orb_{\s}=(D_{\s}\cap \PP\fh)_{\red}$, see loc. cit.
\end{proof}

\begin{scor} 
\label{cor:stabilizer-D6} 
The stabilizer $\W_h\subset \W\cong \fD_6$ of a point $\PP h\in\PP\fh$ under the $\W$-action on $\PP\fh$ is isomorphic to 
\[
\begin{cases}
\ZZ/2\ZZ &\text{if $\PP h\notin \Orb_\ell\cup\Orb_\s\cup\Orb_\rrr$;}
\\
\ZZ/6\ZZ &\text{if $\PP h\in\Orb_\rrr$}.
\end{cases}
\]
In any case $\W_h$ contains the center $\z(\W)\cong \ZZ/2\ZZ$ of $\W\cong \fD_6$.
For the stabilizer $G_h\subset G$ of $\PP h$ we have
\[
G_h\cong 
\begin{cases}
\Gm^2\rtimes\ZZ/2\ZZ &\text{if $\PP h\notin \Orb_\ell\cup\Orb_\s\cup\Orb_\rrr$,}
\\
\Gm^2\rtimes\ZZ/6\ZZ &\text{if $\PP h\in\Orb_\rrr$.}
\end{cases}
\]
In the first case the generator of $\ZZ/2\ZZ$ acts 
on $\Gm^2$ via $g\mapsto g^{-1}$ 
and in the second case the generator of $\ZZ/6\ZZ$ acts 
on $\Gm^2$ via an element of order $6$ unique up to conjugacy in $\Aut(\Gm^2)\simeq \GL_2(\ZZ)$. 
\end{scor}

\begin{proof}
The first assertion follows immediately from Lemma~\ref{lemma:stabilizer-S3}. 
The second follows from Lemma~\ref{lemma:intersections} due to Proposition~\ref{prop:summary}\ref{prop:summary:ii}. 
Notice that $T\subset G_h$, $\W_h\subset G_h$, and $T\cap \W_h=\{1\}$. Thus, $T\rtimes \W_h\subset G_h$.
The $\W$-action on $\Gm^2$ by conjugation 
yields a faithful representation
$\W=\fD_6\hookrightarrow\Aut(\Gm^2)=\GL_2(\ZZ)$. 
According to Maschke's Theorem the center of $\W$ is represented by scalar matrices. Now the last assertions follow.
\end{proof}

\begin{srem}
In the case $\PP h\in \Orb_\ell$ the fourfold $V(h)$ is singular, and in the case $\PP h\in\Orb_\s$ 
its automorphism group is $\GL_2(\Bbbk)\rtimes\ZZ/2\ZZ$, as it was mentioned at the beginning of this section.
\end{srem}

\subsection{Proof of Theorem~\ref{thm:aut0}\ref{thm:aut0:i}-\ref{thm:aut0:ii}} \label{ss:mthm:2}
As we have already mentioned, Theorem~\ref{thm:aut0}\ref{thm:aut0:i}-\ref{thm:aut0:ii} is reduced to Theorem~\xref{mthm:2}.

\begin{proof}[{Proof of Theorem~\xref{mthm:2}}] 
According to Proposition~\ref{lemma:max-torus}\ref{lemma:max-torus:ii}
we may suppose $\Aut^0(V)=T\cong\Gm^2$ and $V=V(h)$ for some $h\in\fh$. In fact, we have 
$\PP h\in\PP\fh\setminus (\Orb_\ell\cup\Orb_\s)$. Indeed, 
\begin{enumerate}
\item[$\bullet$]
the fourfold $V$ is singular for $\PP h\in\Orb_\ell$,
\item[$\bullet$]
$\Aut^0(V)=\GL_2(\Bbbk)$ for $\PP h\in\Orb_\s$, and 
\item[$\bullet$] $\Aut^0(V)=\Gm^2$ for $\PP h\in\PP\fh\setminus (\Orb_\ell\cup\Orb_\s)$,
\end{enumerate}
see Lemma~\ref{lemma:intersections} and Proposition~\ref{prop:aut0}.
By Corollary~\ref{cor:stabilizer-D6} and Remark~\ref{rem:stab}, $\Aut(V)=T\rtimes\ZZ/6\ZZ$ for $h\in\Orb_\rrr$ and $\Aut(V)=T\rtimes\ZZ/2\ZZ$ for 
$h\in \fh\setminus (\Orb_\rrr\cup\Orb_\ell\cup\Orb_\s)$.
Now the first assertion in 
\ref{mthm:2:i} 
follows due to the fact that the points on the same $\W$-orbit in $\PP\fh$ 
correspond to isomorphic Fano--Mukai fourfolds. The first assertion in 
\ref{mthm:2:ii} is immediate. The last assertions in~\ref{mthm:2:i} and~\ref{mthm:2:ii} are provided by
Corollary~\ref{cor:stabilizer-D6}. 
\end{proof}

\subsection{Geometry of the $\Aut(V)$-action on $V$}
Let again $V$ be a Fano--Mukai fourfold of genus 10. 
By Lemma~\ref{lem:cubic-cones:2}, $V$ contains 6, 4, and a one-parameter family of cubic cones if $\Aut^0(V)=\Gm^2, \Ga\times\Gm$, 
and $\GL_2(\Bbbk)$, respectively. In the latter case exactly 2 of the cubic cones on $V$ are $\Aut^0(V)$-invariant, while in the first two cases all of them are. 
In the first case the 6 cones are arranged in a cycle, and in the second case the 4 cones are arranged in a string, see loc. cit. 
The component group $\Aut(V)/\Aut^0(V)$ acts naturally on this cycle and this string, respectively. This action is described in the next proposition.

\begin{prop}
\label{mthm:2:iii} 
For any Fano--Mukai fourfold $V$ of genus $10$, $\Aut(V)\setminus\Aut^0(V)$ contains an involution which
interchanges the members in each pair of disjoint $\Aut^0(V)$-invariant cubic cones on $V$. In the case $\Aut(V)=\Gm^2\rtimes\ZZ/6\ZZ$ 
any generator of $\ZZ/6\ZZ$ acts on the $6$-cycle of cubic cones on $V$ via a cyclic shift of order $6$. 
\end{prop}

\begin{proof}
In the case where $\Aut^0(V)\neq\Gm^2$
the first statement holds by \cite[Theorem~1.3(i),(ii)]{PZ18}. Suppose further that $\Aut^0(V)=T\cong\Gm^2$. Recall that $\Omega$ 
contains exactly 6 fixed points of $T$ and these are the points $\PP\fg_\balpha$, $\balpha\in\Delta_\ell$, see Proposition~\ref{prop:Omega}\ref{fixed-points}. 
On the other hand, the 6 cubic cones on $V$ being $T$-invariant, 
their vertices are fixed by $T$. Therefore, these vertices coincide with the points $\PP\fg_\balpha$, $\balpha\in\Delta_\ell$. 

By Theorem~\ref{thm:aut0} either $\Aut(V)=T\rtimes\ZZ/2\ZZ$, or $\Aut(V)=T\rtimes\ZZ/6\ZZ$. In both cases, the second factor is a cyclic subgroup 
of the Weyl group $\W$ of $G$ and it contains the center of $\W$. The central involution $\iota\in \W$ acts on $\fh^\vee$ via $\balpha\mapsto -\balpha$. 
It defines an automorphism of the Lie algebra $\fg$ which sends $\fg_\balpha$ to $\fg_{-\balpha}$ for any $\balpha\in\Delta_\ell$, 
see \cite[Proposition~14.3]{Hum72}. Hence, $\iota$ induces a fixed point free involution of the cycle of 6 cubic cones. 
Clearly, the latter involution interchanges members in each pair of opposite vertices of the cycle, that is, in each pair of disjoint cubic cones on $V$, 
see Lemma~\ref{lem:cubic-cones:2}\ref{lem:cubic-cones:2:cycle}. 

Let further $\sigma\in \ZZ/6\ZZ\subset \W$ be an element of order 6 from $\Aut(V)$. 
Then $\sigma$ acts on $\fh^\vee$ via a rotation of order 6 and acts on the $6$-cycle of long roots $\balpha\in\Delta_\ell$
via a cyclic shift of order 6. Hence, $\sigma$ induces an automorphism of the Lie algebra $\fg$ permuting cyclically the 6 root subspaces 
$\fg_\balpha$, $\balpha\in\Delta_\ell$. In turn, the induced action on the cycle of cubic cones on $V$ is a cyclic permutation of order 6, cf.~the proof of Proposition~\ref{prop:Omega}\ref{fixed-points}.
\end{proof}


\newcommand{\etalchar}[1]{$^{#1}$}
\def\cprime{$'$}

\end{document}